\documentclass[12pt]{amsart}

\usepackage[margin=1in,includeheadfoot]{geometry}
\usepackage[colorlinks=true]{hyperref}
\usepackage{parskip}
\usepackage{array}
\usepackage{amssymb}

\usepackage{booktabs}

\renewcommand{\emptyset}{\varnothing}

\newcommand{\Size}[1]{\left| #1 \right|}
\newcommand{\Span}[1]{\left< #1 \right>}
\newcommand{\Fix}{\mathrm{Fix}}

\newcommand\inverse{{^{-1}}}
\newcommand\cl{\operatorname{A}}

\newcommand{\partitions}{\vdash}
\newcommand{\rpartitions}[1]{\vdash^{#1}}

\newcommand{\CC}{\mathcal{C}}
\newcommand{\RR}{\mathcal{R}}

\newcommand{\Sym}{\mathfrak{S}}

\newtheorem{Theorem}{Theorem}[section]
\newtheorem{Lemma}[Theorem]{Lemma}
\theoremstyle{definition}

\title[On Reflection Subgroups of Finite Coxeter Groups]
{On Reflection Subgroups of Finite Coxeter Groups}

\subjclass[2010]{20F55 (05E15)}
\keywords{Finite Coxeter groups, reflection subgroups, Coxeter elements}

\author[J.M. Douglass]{J. Matthew Douglass} 
\address{Department of Mathematics,
University of North Texas, 
Denton TX, USA 76203}
\email{douglass@unt.edu} 

\author[G. Pfeiffer]{G\"otz Pfeiffer}
\address{School of Mathematics, Statistics and Applied Mathematics,
NUI Galway, Ireland}
\email{goetz.pfeiffer@nuigalway.ie}

\author[G. R\"ohrle]{Gerhard R\"ohrle}
\address
{Fakult\"at f\"ur Mathematik,
Ruhr-Universit\"at Bochum,
D-44780 Bochum, Germany}
\email{gerhard.roehrle@rub.de}

\begin{document}

\begin{abstract}
  Let $W$ be a finite Coxeter group. We classify the reflection subgroups of
  $W$ up to conjugacy and give necessary and sufficient conditions for the
  map that assigns to a reflection subgroup $R$ of $W$ the conjugacy class
  of its Coxeter elements to be injective, up to conjugacy.
\end{abstract}

\maketitle

\section{Introduction}
\label{sec:intro}

Throughout, let $(W,S)$ be a finite Coxeter system with distinguished set of
generators $S$ and let $E$ be the real reflection representation of
$W$. Define $T=\{\, wsw\inverse\mid w\in W,\,s\in S\,\}$ to be the set of elements of
$W$ that act on $E$ as reflections. By a \emph{reflection subgroup of $W$}
we mean a subgroup of $W$ generated by a subset of $T$.
Reflection subgroups of $W$ play an important role in the theory of Coxeter
groups; for instance, by a fundamental theorem due to Steinberg,
\cite[Thm.~1.5]{Steinberg64}, the stabilizer of any subspace of $E$ is a
reflection subgroup of~$W$.

Our first aim in this note is to give a complete classification of all
reflection subgroups of $W$ up to conjugacy.  In case $W$ is a Weyl group,
Carter \cite[p.~8]{Carter1972} has already outlined a procedure which leads
to the classification based on the algorithm of Borel--De
Siebenthal~\cite{BorelDeSiebenthal1949}. Here, we recast slightly Carter's
construction and give the classification for non-crystallographic Coxeter
groups as well. 
Similar classifications have been described by Felikson and 
Tumarkin~\cite{TumarkinFelikson2005}, and by Dyer and 
Lehrer~\cite{DyerLehrer2009}.
Our methods differ from those used in the sources cited above in that
we use the notion of a parabolic closure of a reflection subgroup as
an inductive tool in our analysis.

Every reflection subgroup of $W$ is a maximal rank reflection subgroup of
some parabolic subgroup of $W$. Thus, classifying conjugacy classes of
reflection subgroups may be done recursively and reduces to first classifying
conjugacy classes of parabolic subgroups and then classifying maximal rank
subgroups of irreducible Coxeter groups. Conjugacy classes of parabolic
subgroups of an irreducible finite Coxeter group are described in Chapter~2
and Appendix~A of \cite{GeckPfeiffer2000}.  Classifying maximal rank
reflection subgroups of $W$ amounts to listing, up to the action of $W$, all
subsets $Y$ of $T$ whose fixed point set in $E$ is trivial and which are
closed in the sense that $\Span{Y} \cap T=Y$. In case $W$ is a Weyl
group, the algorithm of Borel--De Siebenthal~\cite{BorelDeSiebenthal1949} is
computationally much more efficient than classifying subsets of $T$
with the two required properties.

We have implemented the classification algorithms in the computer
algebra system {\sf GAP}~\cite{GAP} with the aid of the package {\sf
  CHEVIE}~\cite{chevie}. Thus, it is feasible to actually compute the
classification explicitly for $W$ of a fixed rank. Indeed, we present
the classification in cases $W$ is a Weyl group of exceptional type,
or a non-crystallographic Coxeter group of type $H_3$ and $H_4$, in
the form of explicit lists.

Our second aim in this note is to study the map which assigns to a
given conjugacy class of reflection subgroups the conjugacy class of
its Coxeter elements. It is well
known~\cite[Lem.~3.5]{OrlikSolomon1983} that if $R$ and $R'$ are
parabolic subgroups containing Coxeter elements $c$ and $c'$
respectively, then $R$ and $R'$ are conjugate subgroups if and only if
$c$ and $c'$ are conjugate in $W$. Thus, conjugacy classes of
parabolic subgroups are parametrized by a distinguished set of
conjugacy classes of elements in $W$.  For general reflection
subgroups this need not be the case.  However, the following somewhat
surprising result shows that when $T$ is a single conjugacy class,
conjugacy classes of reflection subgroups are still parametrized by
the conjugacy classes of their Coxeter elements in all but one case.

\begin{Theorem}
  \label{thm:main}
  Suppose that $T$ is a single conjugacy class.  Let $R$ and $R'$ be
  reflection subgroups containing Coxeter elements $c$ and $c'$,
  respectively. Then $R$ and $R'$ are conjugate if and only if $c$ and
  $c'$ are conjugate in $W$; unless $W$ is of type $E_8$ and $R$ and
  $R'$ are of types $A_1 A_7$ and $A_3 D_5$, respectively.  In this
  case, $c$ and $c'$ are conjugate, while $R$ and $R'$ are not.
\end{Theorem}

This theorem is an immediate consequence of the classification of the
reflection subgroups of $W$ and our computation of the map $\gamma$,
which is defined as follows.  Denote by $\RR$ the set of conjugacy
classes of reflection subgroups of $W$ and by $\CC$ the set of
conjugacy classes of elements of $W$.  Then, denote by
\[
\gamma \colon \RR \to \CC
\]
the map defined by $\gamma([R]) = [c]$, which associates to the conjugacy
class $[R]$ of a reflection subgroup $R$ of $W$ the conjugacy class $[c]$ in
$W$ of a Coxeter element $c$ in $R$.  This map $\gamma$ is well-known to be
a bijection for Coxeter groups $W$ of type $A_n$, with both the
conjugacy classes of reflection subgroups and the conjugacy classes of
elements of $W$ labeled by the partitions of~$n$.  In \S\ref{sec:mick} we
state necessary and sufficient conditions for the map $\gamma$ to be an
injection
(note that if $\gamma$ is an injection, then the statement of Theorem \ref{thm:main} holds).
The image of $\gamma$ is computed explicitly for each type of
irreducible Coxeter group in \S\ref{sec:classical} - \S\ref{sec:brian} and
Tables \ref{tab:e6} - \ref{tab:h4}.  The classes in the image of $\gamma$
are also known in the literature as \emph{semi-Coxeter classes}, e.g.,
see~\cite{CarterElkington1972}.  Properties of the map $\gamma$ have not
been considered in the earlier literature on the subject.

The rest of this note is organized as follows. In \S\ref{sec:mick} we recall
some definitions, give some preliminary results, and state precisely when
the map $\gamma$ is injective or surjective. 
\S\ref{sec:classical} contains explicit combinatorial rules
that describe the map $\gamma$ for classical types, and demonstrate
that $\gamma$ is surjective but not injective for Coxeter groups of
type $B_n$ ($n \geq 2$), and that $\gamma$ is injective but not
surjective for Coxeter groups of type $D_n$ ($n \geq 4$).
The classification of conjugacy classes of
reflection subgroups and the explicit computation of the map $\gamma$ is
given for classical Weyl groups in \S\ref{sec:classical}
(with the examples of $W(B_5)$ and $W(D_6)$ in Tables \ref{tab:b5} and \ref{tab:d6} respectively); 
for exceptional
Weyl groups in \S\ref{sec:keith} and Tables \ref{tab:e6} - \ref{tab:g2}; and for
non-crystallographic Coxeter groups in \S\ref{sec:brian} and Tables
\ref{tab:h3} and \ref{tab:h4}.

\section{Preliminaries}
\label{sec:mick}

For general information on Coxeter groups, root systems, and groups
generated by reflections, we refer the reader to
Bourbaki~\cite{Bourbaki1968}.

For the rest of this note we fix a $W$-invariant, positive definite,
bilinear form on $E$.

Notice first that if $R= \langle Y \rangle$ is a reflection subgroup of
$W$, then $R$ is a Coxeter group in its own right. 
Moreover, the orthogonal complement of the space of fixed points of
$R$ in $E$ is an $R$-stable subspace that affords the reflection
representation of $R$.
 
Recall that a \emph{parabolic subgroup} of $W$ is a subgroup of the form
\[
W_V=\{\, w\in W\mid w(v)=v \ \, \forall v\in V\,\},
\]
where $V$ is a subspace of $E$. By Steinberg's Theorem~\cite[Thm.~1.5]{Steinberg64}, parabolic subgroups
are generated by the reflections they contain and so are reflection
subgroups.

For a subset $X$ of $W$ let
\[
\Fix(X)= \{\, v\in E\mid x(v)=v\  \, \forall x\in X\,\}
\]
denote the set of fixed points of $X$ in $E$. 
Following Solomon~\cite{Solomon1976} and Bergeron et al.~\cite{BergeronEtAl1992},
we define the \emph{parabolic closure}
of $X$ to be the parabolic subgroup $\cl (X)= W_{\Fix(X)}$ of~$W$. Obviously
$X\subseteq \cl (X)$ and it follows from Steinberg's Theorem that $\cl(\cl(
X)) = \cl (X)$. When $X=\{w\}$ we simply write $\Fix(w)$ and $\cl(w)$ instead
of $\Fix(\{w\})$ and $\cl(\{w\})$, respectively.
For a discussion of parabolic closures of finitely generated
subgroups of arbitrary Coxeter systems, see the recent paper by Dyer~\cite{Dyer2010}.

For $w, x\in W$ we denote the 
$w$-conjugate $w^{-1}xw$ of $x$ by $x^w$ and for a subset $X$ of $W$
let $X^w = \{x^w \mid x \in X\}$ denote the $w$-conjugate of $X$.

The \emph{rank} of a Coxeter group is the cardinality of a Coxeter
generating set, or equivalently, the dimension of its reflection
representation. It follows from the next lemma that every reflection
subgroup is a maximal rank reflection subgroup of its parabolic closure.

\begin{Lemma}
  \label{lemma:rank}
  Let $R$ be a reflection subgroup of $W$.  Then $R$ and its parabolic
  closure $\cl(R)$ have the same rank as Coxeter groups.
\end{Lemma}

\begin{proof}
  The rank of $R$ is the codimension of its fixed point space
  $\Fix(R)$. The rank of $\cl(R)$, as stabilizer of $\Fix(R)$, is not
  larger than the rank of $R$, and, since $R \subseteq \cl(R)$, not
  smaller than the rank of $R$ either.
\end{proof}

As noted in the Introduction, the classification of conjugacy classes of
reflection subgroups of $W$ reduces to (1) classifying conjugacy classes of
parabolic subgroups of $W$ and (2) classifying maximal rank reflection
subgroups of irreducible Coxeter groups.

The conjugacy classes of parabolic subgroups of an irreducible finite
Coxeter group are described in Chapter~2 and Appendix~A of
\cite{GeckPfeiffer2000} (see also \cite[Prop.\ 6.3]{BalaCarter1976}). In
most cases, two parabolic subgroups are conjugate if and only if they have
the same type. However, in type $D_{2m}$ there are two conjugacy classes of
parabolic subgroups of type $A_{k_1} \times A_{k_2} \times \cdots \times
A_{k_r}$ with all $k_i$ odd so that $2m = \sum (k_i +1)$ and in type $E_7$
there are two classes of parabolic subgroups for each of the types 
$A_1^3$, $A_1 A_3$, and $A_5$.

For a given $W$, classifying the maximal rank reflection subgroups of $W$ up
to conjugacy amounts to listing, up to conjugacy in $W$, all subsets $Y$
of $T$ such that
\[
\Span{Y} \cap T=Y \qquad \text{and} \qquad \Fix(Y)=\{0\}.
\]
For a Coxeter group of small rank (including the non-crystallographic
types $H_3$ and $H_4$) these sets can be systematically enumerated. As
described below, for crystallographic Coxeter groups, that is, Weyl groups,
using the algorithm of Borel--De Siebenthal~\cite{BorelDeSiebenthal1949} is
computationally more efficient than classifying subsets of $T$.

By a \emph{root system} in $E$ we mean a reduced root system in the
sense of Bourbaki~\cite[Ch.~VI]{Bourbaki1968}. Suppose $\Phi$ is a root
system in $E$. The Weyl group of $\Phi$, $W(\Phi)$, is the group of linear
transformations of $E$ generated by the reflections through the hyperplanes
orthogonal to the roots in $\Phi$. The \emph{dual} of $\Phi$ is the root
system $\tilde \Phi=\{\ \frac 1 {|\alpha|^2} \alpha\mid \alpha \in
\Phi\,\}$. Note that $W(\Phi) = W(\tilde \Phi)$. By a \emph{Weyl group} or a
\emph{crystallographic Coxeter group} we mean the Weyl group of a root
system in $E$.

Suppose that $W=W(\Phi)=W(\tilde\Phi)$ is a Weyl group. We may extract a
classification of the maximal rank reflection subgroups of $W$ from the
arguments in \cite{Carter1972}. Each maximal rank reflection subgroup of $W$
is again a Weyl group and thus is the Weyl group of a maximal rank subsystem
of $\Phi$ or $\tilde \Phi$. By work of Dynkin, two maximal rank subsystems are
isomorphic if and only if they are equivalent under the action of $W$; see
\cite[Prop.\ 32]{Carter1972} or \cite[Ch.~VI, \S 4, ex.~4]{Bourbaki1968}. By
the classification of root systems, two root systems are isomorphic if and
only if they have the same Dynkin diagram. We have already observed that a
root system and its dual have the same Weyl group. Thus, the conjugacy
classes of maximal rank reflection subgroups of $W$ are in one-one
correspondence with the set of Coxeter graphs arising from Dynkin diagrams
of maximal rank subsystems of $\Phi$ and $\tilde \Phi$.

The Borel--De Siebenthal algorithm produces all maximal rank subsystems of
$\Phi$ and $\tilde \Phi$ as follows (see~\cite[p.~8]{Carter1972}).
\begin{enumerate}
\item Add a node to the Dynkin diagram of $\Phi$ corresponding to the
  negative of the highest root of $\Phi$. Take the extended Dynkin diagram
  and remove one node in all possible ways.
\item Add a node to the Dynkin diagram of $\tilde \Phi$ corresponding to the
  negative of the highest root of $\tilde \Phi$. Take the extended Dynkin
  diagram and remove one node in all possible ways. 
\item Repeat steps (1) and (2) with each of the resulting Dynkin diagrams
  until no new diagrams appear.
\end{enumerate}

This algorithm does not apply to the non-crystallographic groups $W(H_3)$,
$W(H_4)$ and $W(I_2(m))$, but these groups are sufficiently small that all relevant information can be calculated directly

We now turn to the map $\gamma$ which assigns to a given conjugacy class of
reflection subgroups of $W$ the conjugacy class of its Coxeter elements.

Recall~\cite[Ch.~V, \S6, no.~1]{Bourbaki1968} that a \emph{Coxeter
  element} in $W$ is the product of the elements of some Coxeter
generating set of $W$ taken in some order.  All Coxeter elements of
$W$ are conjugate in $W$.

Suppose that $R$ is a reflection subgroup of $W$. Then $R$ is a Coxeter
group and so we may consider Coxeter elements in $R$. If $c$ is a
Coxeter element in $R$ and $w$ is in $W$, it is easy to see that $c^w$
is a Coxeter element in $R^w$. Thus, conjugate reflection subgroups of $W$
have conjugate Coxeter elements and the map $\gamma$ is well-defined.

The proof of Theorem \ref{thm:main} follows immediately from the
classification of reflection subgroups and the explicit computation of the
map $\gamma$ in Theorems \ref{thm:a} and \ref{thm:d} and Tables \ref{tab:e6}
- \ref{tab:e8}. It would be interesting to have a conceptual explanation of
why the single exception occurs in Theorem \ref{thm:main}.  More generally,
from Theorems \ref{thm:a}, \ref{thm:b}, and \ref{thm:d} along with Tables
\ref{tab:e6} - \ref{tab:h4}, we derive necessary and sufficient conditions
for the map $\gamma$ to be injective.

\begin{Theorem}
  Suppose that $W$ is irreducible and not of type $E_8$. Then the map
  $\gamma \colon \RR\to \CC$ is injective if and only if $T$ is a single
  conjugacy class in~$W$. 

  If $W$ is of type $E_8$, then $\gamma$ is not injective:  the
  conjugacy classes of reflection subgroups of types $A_1 A_7$ and $A_3
  D_5$ both map to the same conjugacy class of elements of $W$.
\end{Theorem}

Hence, we conclude that the map $\gamma$
 is injective if and only if $W$ has type $A_n$, $D_n$,
$E_6$, $E_7$, $H_3$, $H_4$ and $I_2(m)$ with $m$ odd. Moreover, it follows from the
computations in \S \ref{sec:classical} - \S \ref{sec:brian} that, if $W$ is irreducible, then 
$\gamma$ is surjective when $W$  has type $A_n$, $B_n$, or $G_2$. Notice that when the map 
$\gamma$  is surjective, every
conjugacy class of $W$ contains a representative that is a Coxeter element
in some reflection subgroup of $W$.

It is easy to see that $\cl(X^w)= \cl(X)^w$ when $X\subseteq W$ and $w\in
W$. In particular, conjugate reflection subgroups have conjugate parabolic
closures. Similarly, conjugate elements in $W$ have conjugate parabolic
closures. In particular, if $c\in R$ and $c'\in R'$ are Coxeter
elements in reflection subgroups $R$ and $R'$, and $c$ and $c'$ are
conjugate in $W$, then $\cl(c)$ and $\cl(c')$ are conjugate in~$W$.

\begin{Lemma} 
  \label{la:7}
  Suppose $R$ is a reflection subgroup of $W$ and $x\in R$ is not contained
  in any proper parabolic subgroup of $R$. Then $\cl(x)=
  \cl(R)$. Consequently, if $V$ is a subspace of $E$ and $c$ is a
  Coxeter element of $R$ that is conjugate to an element of $W_V$, then
  $R$ is conjugate to a subgroup of $W_V$.
\end{Lemma}

\begin{proof}
  It is shown in \cite[\S2]{Carter1972} that
  $\Fix(x)=\Fix(R)$. It follows immediately that $\cl(x)= \cl(R)$.

  For the second statement, let $w \in W$ be such that $c^w$ is in
  $W_V$. Then $\cl(c^w) \subseteq W_V$ and so $R^w\subseteq \cl(R)^w
  =\cl(c)^w= \cl(c^w) \subseteq W_V$.
\end{proof}

Now suppose that $R$ and $R'$ are reflection subgroups of $W$ containing
Coxeter elements $c$ and $c'$ respectively. Then, if $c$ and $c'$ are
conjugate in $W$, $\cl(R)$ and $\cl(R')$ are conjugate.  In other words,
even if $\gamma$ is not injective, reflection subgroups with non-conjugate
parabolic closures must have non-conjugate Coxeter elements. This
observation shows that conjugacy classes containing Coxeter elements of
reflection subgroups are separated by the parabolic closures of reflection
subgroups that contain them. 

Note that Lemma~\ref{la:7} generalizes \cite[Lem.~7]{Solomon1976}
which is the special case of Lemma \ref{la:7} when $R$ is a parabolic
subgroup of $W$.  In the same way, Theorem~\ref{thm:main} generalizes
the forward implication of~\cite[Lem.~3.5]{OrlikSolomon1983}.

\section{The classical Weyl groups}
\label{sec:classical}

A \emph{partition} $\lambda = (\lambda_1, \dots, \lambda_k)$ is a
non-increasing finite sequence of positive integers $\lambda_1 \geq
\dots \geq \lambda_k > 0$.  The integers $\lambda_i$ are called the
\emph{parts} of the partition $\lambda$.  If $\sum_{i=1}^k \lambda_i =
n$, then $\lambda$ is a partition of $n$ and we write
$\lambda \partitions n$.  The unique partition of $n = 0$ is the \emph{empty
partition}, denoted by $\emptyset$.  We denote by $\ell(\lambda) = k$
the length of the partition $\lambda = (\lambda_1, \dots, \lambda_k)$,
e.g., $\ell(\emptyset) = 0$.
A partition of $n$ is \emph{even}, if all its parts are even, i.e., if
it has the form $\lambda = (2 \mu_1, \dots, 2 \mu_k)$ for some
partition $\mu$ of $n/2$.  The \emph{join} $\lambda^1 \cup \lambda^2$
of two partitions $\lambda^1 \partitions n_1$ and
$\lambda^2 \partitions n_2$ is the partition of $n_1 + n_2$ consisting
of the parts of both $\lambda^1$ and $\lambda^2$, suitably arranged.
The \emph{sum} of a partition $\lambda = (\lambda_1, \dots,
\lambda_k)$ and an integer $m$ is the partition $\lambda + m =
(\lambda_1 + m, \dots, \lambda_k+m)$.  We write $\lambda > m$ if
$\lambda_i > m$ for all parts $\lambda_i$ of $\lambda$.
Note that, vacuously, $\emptyset > m$ for all $m$.

The symmetric group $\Sym_n$ on $n$ points is a Coxeter group of type
$A_{n-1}$ with Coxeter generators $s_i = (i, i+1)$, $i = 1, \dots, n-1$.
The \emph{cycle type} of $w \in \Sym_n$ is the partition $\lambda$ of $n$
which has the lengths of the cycles of $w$ on $\{1, \dots, n\}$ as its parts
(here a fixed point contributes a cycle of length $1$). Of course, two
permutations in $\Sym_n$ are conjugate if and only if they have the same
cycle type. The next theorem is well-known, and can easily be deduced from
Bourbaki~\cite[Ch.~VI, \S 4, ex.~4]{Bourbaki1968}.

\begin{Theorem}
\label{thm:a}
Let $W$ be a Coxeter group of type $A_n$. Then every reflection subgroup of
$W$ is a parabolic subgroup. Moreover, the map $\gamma$ from conjugacy
classes of reflection subgroups to conjugacy classes of $W$ is a
bijection. Both sets are in one-to-one correspondence with the set of all
partitions of $n$.
\end{Theorem}

An \emph{$r$-partition} is a sequence $\lambda = (\lambda^1, \dots,
\lambda^r)$ of $r$ partitions $\lambda^1, \dots, \lambda^r$.  We say that
$\lambda$ is an $r$-partition of the integer $n$, and write $\lambda
\rpartitions{r} n$, if $\lambda^1 \cup \dots \cup \lambda^r \partitions n$.
We call $\lambda$ a \emph{double partition} if $r = 2$, and a \emph{triple
  partition} if $r = 3$.

The Coxeter group $W(B_n)$ acts faithfully as a group of signed permutations
on the set of long roots $\{\pm e_i \mid i = 1, \dots, n\}$, permuting the
lines $\Span{e_i}$, $i = 1, \dots, n$.  A cycle of $w$ in $W(B_n)$ is either
\emph{positive} or \emph{negative}, depending on whether the number of
positive roots $e_i$ with $\Span{e_i}$ in the cycle that are mapped to 
negative roots is even or odd.  The \emph{cycle type} of $w$ in $W(B_n)$ is a
double partition $\lambda = (\lambda^1, \lambda^2)$ of $n$, where
$\lambda^1$ records the lengths of the positive cycles of $w$ and
$\lambda^2$ records the lengths of the negative cycles. Again, two elements
of $W(B_n)$ are conjugate if and only if they have the same cycle type and,
in this way, the double partitions of $n$ naturally parametrize the conjugacy
classes of $W(B_n)$.

According to \cite[Prop.~2.3.10]{GeckPfeiffer2000}, the parabolic
subgroups of $W(B_n)$ are of the form $W(B_{n-m}) \times \prod_i
W(A_{\lambda_i-1})$, one conjugacy class for each partition
$\lambda \partitions m$, $0 \leq m \leq n$.  By Borel--De Siebenthal,
the maximal rank reflection subgroups of $W(B_n)$ are of type $\prod_i
W(B_{\lambda^1_i}) \times \prod_i W(D_{\lambda^2_i})$, one class for
each double partition $\lambda \partitions^2 n$ with $\lambda^2 > 1$
(or $\lambda^2 = \emptyset$).  It follows that the reflection
subgroups of $W(B_n)$ are direct products of Coxeter groups of types
$A$, $B$ and $D$, and their classes are naturally labeled by triple
partitions of~$n$.


\begin{Theorem}
  \label{thm:b}
  Let $W$ be a Coxeter group of type $B_n$, $n \ge 2$.  Then the conjugacy classes of
  reflection subgroups of $W$ are represented by
  \begin{align*}
    \{W_{\lambda} \mid \lambda \rpartitions{3} n,\, \lambda^3 > 1\},
  \end{align*}
  where $W_{\lambda} = \prod_i W(A_{\lambda^1_i-1}) \times \prod_i
  W(B_{\lambda^2_i}) \times \prod_i W(D_{\lambda^3_i})$.  The parabolic
  closure of $W_{\lambda}$ has type $W(B_{n-m}) \times \prod_i
  W(A_{\lambda^1_i-1})$, where $\lambda^1 \partitions m$. The Coxeter
  elements of $W_{\lambda}$ have cycle type $(\lambda^1, \lambda^2 \cup
  (\lambda^3-1) \cup 1^{\ell(\lambda^3)})$.  In particular, the map $\gamma
  \colon \RR \to \CC$ is surjective, but not injective.
\end{Theorem}

We illustrate the classification in type $B_n$ in Table \ref{tab:b5} below, 
where
we list all conjugacy classes of reflection subgroups of $W(B_5)$ according
to Theorem \ref{thm:b}.  Clearly, it follows from the data in Table
\ref{tab:b5} that $\gamma$ is not injective in this case.
 
The Coxeter group $W(D_n)$ is a normal subgroup of index $2$ in $W(B_n)$,
and as such it is a union of $W(B_n)$-conjugacy classes of elements.  In
fact, the class of elements of cycle type $\lambda = (\lambda^1, \lambda^2)$
is contained in $W(D_n)$ if and only if $\ell(\lambda^2)$ is even, and it is
a single conjugacy class in $W(D_n)$, unless $\lambda^2 = \emptyset$ and
$\lambda^1$ is even.  In the latter case, the $W(B_n)$-class splits into two
$W(D_n)$-classes, labelled $(\lambda^1, +)$ and $(\lambda^1,-)$.  In this
way, the conjugacy classes of $W(D_n)$ are parametrized by certain 
double partitions
of $n$.

According to \cite[Prop.~2.3.13]{GeckPfeiffer2000}, $W(D_n)$ has three
distinct kinds of parabolic subgroups: one class of subgroups of type
$W(D_{n-m}) \times \prod_i W(A_{\lambda_i-1})$ for each partition
$\lambda \partitions m$, $0 \leq m \leq n - 2$, two classes of subgroups of
type $\prod_i W(A_{\lambda_i-1})$ for each even partition
$\lambda \partitions n$, and one class of subgroups of type $\prod_i
W(A_{\lambda_i-1})$ for each non-even partition $\lambda \partitions n$.
By Borel--De Siebenthal, the maximal rank reflection
subgroups of $W(D_n)$ are of type $\prod_i W(D_{\lambda_i})$, one class for
each partition $\lambda \partitions n$ with $\lambda > 1$.
It follows that
reflection subgroups of $W(D_n)$ are direct products of Coxeter groups of
types $A$ and $D$, and their classes are naturally labeled by double
partitions of~$n$.
 This yields the
following classification of the conjugacy classes of reflection subgroups of
$W(D_n)$, in terms of double partitions of $n$.

\begin{Theorem}
\label{thm:d}
Let $W$ be a Coxeter group of type $D_n$, $n \geq 4$.  Then the conjugacy
classes of reflection subgroups of $W$ are represented by
\begin{align*}
  \{W_{\lambda} \mid \lambda \rpartitions{2} n,\, \lambda^2 > 1\}
\end{align*}
if $n$ is odd, and by
\begin{align*}
  \{W_{\lambda} \mid \lambda \rpartitions{2} n,\, \lambda^2 > 1 \text{ and }
  \lambda^1 \text{ non-even in case } \lambda^2 = \emptyset\} \cup
  \{W_{\lambda}^{\pm} \mid \lambda \partitions n \text{ and } \lambda \text{ even}\}
\end{align*}
if $n$ is even, where $W_{\lambda} = \prod_i W(A_{\lambda^1_i-1}) \times
\prod_i W(D_{\lambda^2_i})$ and $W_{\lambda}^{\epsilon} = \prod_i
W(A_{\lambda_i-1})$, where $\epsilon = \pm$.
The parabolic closure of $W_{\lambda}$ has type
$W(D_{n-m}) \times \prod_i W(A_{\lambda^1_i-1})$, where
$\lambda^1 \partitions m$; the parabolic closure of $W_{\lambda}^{\epsilon}$
is $W_{\lambda}^{\epsilon}$ itself.  The Coxeter elements of $W_{\lambda}$
have cycle type $(\lambda^1, (\lambda^2-1) \cup 1^{\ell(\lambda^2)})$; the
Coxeter elements of $W_{\lambda}^{\epsilon}$ have cycle type $(\lambda,
\epsilon)$.  In particular, the map $\gamma \colon \RR \to \CC$ is
injective, but not surjective.
\end{Theorem}

We illustrate the classification in type $D_n$ from Theorem \ref{thm:d} for
$n = 6$ in Table \ref{tab:d6} below.

\section{The exceptional Weyl groups}
\label{sec:keith}
For the exceptional Weyl groups all results are obtained by following the
recursive procedure outlined in \S\ref{sec:mick}, using the Borel--De
Siebenthal algorithm for the various factors of each standard parabolic
subgroup of $W$. The calculations were carried out with the use of {\sf
  GAP}~\cite{GAP} and {\sf CHEVIE}~\cite{chevie}. Here the conjugacy classes
of the elements in $W$ are labeled as in Carter~\cite{Carter1972}.

In Tables \ref{tab:e6} - \ref{tab:g2} we list all reflection subgroups in
case $W$ is of exceptional type up to conjugacy.  In the cases when $W$ has
only a single class of reflections, it is readily checked that $\gamma$ is
injective, as required for Theorem \ref{thm:main}.

Table \ref{tab:e8} contains the results for $W(E_8)$. Here the two maximal
rank reflection subgroups of types $A_1 A_7$ and $A_3 D_5$ have Coxeter
elements that are conjugate in $W$.  Hence $\gamma$ is not injective.

In Tables \ref{tab:f4} and \ref{tab:g2} we list all conjugacy classes of
reflection subgroups of $W(F_4)$ and $W(G_2)$, respectively. In both
instances we see that $\gamma$ is not injective.

\section{The non-crystallographic cases}
\label{sec:brian}
As in the exceptional cases, the non-crystallographic instances were
computed using {\sf GAP}~\cite{GAP} and {\sf CHEVIE}~\cite{chevie}. The
Borel--De Siebenthal algorithm does not apply, but these groups are
sufficiently small that all the relevant information can be calculated
directly.

In Tables \ref{tab:h3} and \ref{tab:h4} we list all conjugacy classes of
reflection subgroups of $W(H_3)$ and $W(H_4)$, respectively. Here we see
that $\gamma$ is injective in both cases.  The labeling of the conjugacy
classes is the one used by {\sf CHEVIE}.

The reflection subgroups of the dihedral group $W(I_2(m))$ can be described
as follows.

\begin{Theorem}\label{thm:i2}
  Let $W$ be of type $I_2(m)$, $m = 5$  or $m >6$.

  \begin{enumerate}
  \item If $m$ is odd then the classes of reflection subgroups of $W$ are of
    types $\emptyset$, $A_1$, and $I_2(d)$ where $d > 1$ is a divisor of
    $m$.  The map $\gamma$ is injective, but not surjective.
  \item If $m$ is even then the classes of reflection subgroups of $W$ are
    of types $\emptyset$, $A_1$, $\tilde{A}_1$, $I_2(d)$ where $d > 1$ is a
    divisor of $m$ and $\tilde{I}_2(d)$ where $2d > 2$ is a divisor of
    $m$. The map $\gamma$ is neither injective nor surjective.
  \end{enumerate}
\end{Theorem}

The subgroups of a dihedral group are determined by a
straightforward computation. The theorem follows by filtering
out those subgroups that are generated by reflections.

\section{Tables}

In Tables \ref{tab:b5} - \ref{tab:h4} we present the classification of the
reflection subgroups of $W$ in various cases. The tables provide the
following information. In the first column of each table we list the types
of the reflection subgroups $R$ of $W$. In the second column in Tables
\ref{tab:b5} and \ref{tab:d6} we also give the partition representing $R$
according to Theorems \ref{thm:b} and \ref{thm:d}, respectively.  The next
two columns give the cardinality of $R$ and the cardinality of the class
$[R]$ of $R$ (that is, $|W:N_W(R)|$). Finally, in the last column we list the
image of $\gamma$, i.e.\ the class $[c]$ of a Coxeter element $c$ of $R$ in
$W$. For the classical types, conjugacy classes are labeled by cycle
type. For the exceptional types, conjugacy classes are labeled as in
Carter's classification~\cite{Carter1972}.

Conjugacy classes of reflection subgroups with distinct parabolic closures
are separated by horizontal lines. For a given parabolic subgroup $P$ of
$W$, the row for $P$ is preceded by a horizontal line and followed by the
rows for reflection subgroups $R$ of $W$ with $\cl(R) = P$.

\begin{table}[p] 
  \bigskip
  \extrarowheight2pt
  \caption{Reflection subgroups of $W(B_5)$.}
  \label{tab:b5}
  \begin{tabular}[t]{ccrrc}\toprule 
    Type of $R$ & $\lambda$ & $\Size{R}$ & $\Size{[R]}$ & Class
    \\[5pt] \toprule
    $\emptyset$ & $1^5..$& 1 & 1 & $1^5.$\\
    \noalign{\vglue 2pt} \hline
    $B_1$ & $1^4.1.$ & 2 & 5 & $1^4.1$\\
    \noalign{\vglue 2pt} \hline
    $A_1$ & $21^3..$ & 2 & 20 & $21^3.$\\
    \noalign{\vglue 2pt} \hline
    $B_1 A_1$ & $21^2.1.$ & 4 & 60 & $21^2.1$\\
    \noalign{\vglue 2pt} \hline
    $A_1^{2}$ & $2^21..$ & 4 & 60 & $2^21.$\\
    \noalign{\vglue 2pt} \hline
    $A_2$ & $31^2..$ & 6 & 40 & $31^2.$\\
    \noalign{\vglue 2pt} \hline
    $B_2$ & $1^3.2.$ & 8 & 10 & $1^3.2$\\
    $B_1^{2}$ & $1^3.1^2.$ & 4 & 10 & $1^3.1^2$\\
    $D_2$ & $1^3..2$ & 4 & 10 & $1^3.1^2$\\
    \noalign{\vglue 2pt} \hline
    $B_1 A_1^{2}$ & $2^2.1.$ & 8 & 60 & $2^2.1$\\
    \noalign{\vglue 2pt} \hline
    $B_1 A_2$ & $31.1.$ & 12 & 80 & $31.1$\\
    \noalign{\vglue 2pt} \hline
    $A_1 A_2$ & $32..$ & 12 & 80 & $32.$\\
    \noalign{\vglue 2pt} \hline
    $B_2 A_1$ & $21.2.$ & 16 & 60 & $21.2$\\
    $B_1^{2} A_1$ & $21.1^2.$ & 8 & 60 & $21.1^2$\\
    $D_2 A_1$ & $21..2$ & 8 & 60 & $21.1^2$\\
    \noalign{\vglue 2pt} \hline
    $A_3$ & $41..$ & 24 & 40 & $41.$\\
    \noalign{\vglue 2pt} \hline
    $B_3$ & $1^2.3. $ & 48 & 10 & $1^2.3$\\
    $B_1 B_2$ & $1^2.21.$ & 16 & 30 & $1^2.21$\\
    $D_3$ & $1^2..3$ & 24 & 10 & $1^2.21$\\
    $D_2 B_1$ & $1^2.1.2$ & 8 & 30 & $1^2.1^3$\\
    $B_1^{3}$ & $1^2.1^3.$ & 8 & 10 & $1^2.1^3$\\
    \noalign{\vglue 2pt} \hline
    $B_2 A_2$ & $3.2.$ & 48 & 40 & $3.2$\\
    $B_1^{2} A_2$ & $3.1^2.$ & 24 & 40 & $3.1^2$\\
    $D_2 A_2$ & $3..2$ & 24 & 40 & $3.1^2$\\
    \noalign{\vglue 2pt} \hline
    $B_1 A_3$ & $4.1.$ & 48 & 40 & $4.1$\\
    \noalign{\vglue 2pt} \hline
    $B_3 A_1$ & $2.3.$ & 96 & 20 & $2.3$\\
    $B_1 B_2 A_1$ & $2.21.$ & 32 & 60 & $2.21$\\
    $D_3 A_1$ & $2..3$ & 48 & 20 & $2.21$\\
    $D_2 B_1 A_1$ & $2.1.2$ & 16 & 60 & $2.1^3$\\
    $B_1^{3} A_1$ & $2.1^3.$ & 16 & 20 & $2.1^3$\\
    \bottomrule
  \end{tabular}
  \qquad
  \begin{tabular}[t]{ccrrc}\toprule
    Type of $R$ & $\lambda$ & $\Size{R}$ & $\Size{[R]}$ & Class
    \\[5pt] \toprule
    $A_4$ & $5..$ & 120 & 16 & $5.$\\
    \noalign{\vglue 2pt} \hline
    $B_4$ & $1.4.$ & 384 & 5 & $1.4$\\
    $B_1 B_3$ & $1.31.$ & 96 & 20 & $1.31$\\
    $B_2^{2}$ & $1.2^2.$ & 64 & 15 & $1.2^2$\\
    $D_4$ & $1..4$ & 192 & 5 & $1.31$\\
    $D_3 B_1$ & $1.1.3$ & 48 & 20 & $1.21^2$\\
    $D_2 B_2$ & $1.2.2$ & 32 & 30 & $1.21^2$\\
    $B_1^{2} B_2$ & $1.21^2.$ & 32 & 30 & $1.21^2$\\
    $D_2 B_1^{2}$ & $1.1^2.2$ & 16 & 30 & $1.1^4$\\
    $B_1^{4}$ & $1.1^4.$ & 16 & 5 & $1.1^4$\\
    $D_2^{2}$ & $1..2^2$ & 16 & 15 & $1.1^4$\\
    \noalign{\vglue 2pt} \hline
    $B_5$ & $.5.$ & 3840 & 1 & $.5$\\
    $B_1 B_4$ & $.41.$ & 768 & 5 & $.41$\\
    $B_2 B_3$ & $.32.$ & 384 & 10 & $.32$\\
    $D_5$ & $..5$ & 1920 & 1 & $.41$\\
    $D_4 B_1$ & $.1.4$ & 384 & 5 & $.31^2$\\
    $D_3 B_2$ & $.2.3$ & 192 & 10 & $.2^21$\\
    $D_2 B_3$ & $.3.2$ & 192 & 10 & $.31^2$\\
    $B_1^{2} B_3$ & $.31^2.$ & 192 & 10 & $.31^2$\\
    $B_1 B_2^{2}$ & $.2^21.$ & 128 & 15 & $.2^21$\\
    $D_3 B_1^{2}$ & $.1^2.3$ & 96 & 10 & $.21^3$\\
    $D_2 B_1 B_2$ & $.21.2$ & 64 & 30 & $.21^3$\\
    $B_1^{3} B_2$ & $.21^3.$ & 64 & 10 & $.21^3$\\
    $D_2 B_1^{3}$ & $.1^3.2$ & 32 & 10 & $.1^5$\\
    $D_2^{2} B_1$ & $.1.2^2$ & 32 & 15 & $.1^5$\\
    $D_2 D_3$ & $..23$ & 96 & 10 & $.21^3$\\
    $B_1^{5}$ & $.1^5.$ & 32 & 1 & $.1^5$\\
    \bottomrule
  \end{tabular}
  \bigskip
\end{table}

\begin{table}[p]
  \bigskip
  \small
  \extrarowheight1pt
  \caption{Reflection subgroups of $W(D_6)$.}
  \label{tab:d6}
  \begin{tabular}[t]{ccrrc}\toprule
    Type of $R$ & $\lambda$ & $\Size{R}$ & $\Size{[R]}$ & Class
    \\[4pt] \toprule
    $\emptyset$ & $1^6.$ & 1 & 1 & $1^6.$\\
    \noalign{\vglue 1pt} \hline
    $A_1$ & $21^4.$ & 2 & 30 & $21^4.$\\
    \noalign{\vglue 1pt} \hline
    $D_2$ & $1^4.2$ & 4 & 15 & $1^4.1^2$\\
    \noalign{\vglue 1pt} \hline
    $A_1^{2}$ & $2^21^2.$ & 4 & 180 & $2^21^2.$\\
    \noalign{\vglue 1pt} \hline
    $A_2$ & $31^3.$ & 6 & 80 & $31^3.$\\
    \noalign{\vglue 1pt} \hline
    $A_1^{3}$ & $2^3.+$ & 8 & 60 & $2^3.+$\\
    \noalign{\vglue 1pt} \hline
    $A_1^{3}$ & $2^3.-$ & 8 & 60 & $2^3.-$\\
    \noalign{\vglue 1pt} \hline
    $D_2 A_1$ & $21^2.2$ & 8 & 180 & $21^2.1^2$\\
    \noalign{\vglue 1pt} \hline
    $A_1 A_2$ & $321.$ & 12 & 480 & $321.$\\
    \noalign{\vglue 1pt} \hline
    $D_3$ & $1^3.3$ & 24 & 20 & $1^3.21$\\
    \noalign{\vglue 1pt} \hline
    $A_3$ & $41^2.$ & 24 & 120 & $41^2.$\\
    \noalign{\vglue 1pt} \hline
    $D_2 A_1^{2}$ & $2^2.2$ & 16 & 180 & $2^2.1^2$\\
    \noalign{\vglue 1pt} \hline
    $D_2 A_2$ & $31.2$ & 24 & 240 & $31.1^2$\\
    \noalign{\vglue 1pt} \hline
    $A_2^{2}$ & $33.$ & 36 & 160 & $33.$\\
    \noalign{\vglue 1pt} \hline
    $D_3 A_1$ & $21.3$ & 48 & 120 & $21.21$\\
    \noalign{\vglue 1pt} \hline
    $A_1 A_3$ & $42.+$ & 48 & 120 & $42.+$\\
    \noalign{\vglue 1pt} \hline
    $A_1 A_3$ & $42.-$ & 48 & 120 & $42.-$\\
    \bottomrule
  \end{tabular}
  \qquad
  \begin{tabular}[t]{ccrrc}\toprule
    Type of $R$ & $\lambda$ & $\Size{R}$ & $\Size{[R]}$ & Class
    \\[4pt]\toprule
    $A_4$ & $51.$ & 120 & 96 & $51.$\\
    \noalign{\vglue 1pt} \hline
    $D_4$ & $1^2.4$ & 192 & 15 & $1^2.31$\\
    $D_2^{2}$ & $1^2.2^2$ & 16 & 45 & $1^2.1^4$\\
    \noalign{\vglue 1pt} \hline
    $D_2 A_3$ & $4.2$ & 96 & 120 & $4.1^2$\\
    \noalign{\vglue 1pt} \hline
    $D_3 A_2$ & $3.3$ & 144 & 80 & $3.21$\\
    \noalign{\vglue 1pt} \hline
    $D_4 A_1$ & $2.4$ & 384 & 30 & $2.31$\\
    $D_2^2 A_1$ & $2.2^2$ & 32 & 90 & $2.1^4$\\
    \noalign{\vglue 1pt} \hline
    $A_5$ & $6.+$ & 720 & 16 & $6.+$\\
    \noalign{\vglue 1pt} \hline
    $A_5$ & $6.-$ & 720 & 16 & $6.-$\\
    \noalign{\vglue 1pt} \hline
    $D_5$ & $1.5$ & 1920 & 6 & $1.41$\\
    $D_2 D_3$ & $1.32$ & 96 & 60 & $1.21^3$\\
    \noalign{\vglue 1pt} \hline
    $D_6$ & $.6$ & 23040 & 1 & $.51$\\
    $D_2 D_4$ & $.42$ & 768 & 15 & $.31^3$\\
    $D_3^{2}$ & $.3^2$ & 576 & 10 & $.2^21^2$\\
    $D_2^{3}$ & $.2^3$ & 64 & 15 & $.1^6$\\
    \bottomrule
  \end{tabular}
  \bigskip
\end{table}

\begin{table}[p]
  \bigskip
  \small
  \extrarowheight1pt
  \caption{Reflection subgroups of $W(E_6)$.}
  \label{tab:e6}
  \begin{tabular}[t]{crrc}\toprule
    Type of $R$ & $\Size{R}$ & $\Size{[R]}$ & Class 
    \\[4pt]
 \toprule
    $\emptyset$ & 1 & 1 & $1$\\
    \noalign{\vglue 1pt} \hline
    $A_1$ & 2 & 36 & $A_1$\\
    \noalign{\vglue 1pt} \hline
    $A_1^{2}$ & 4 & 270 & $2A_1$\\
    \noalign{\vglue 1pt} \hline
    $A_2$ & 6 & 120 & $A_2$\\
    \noalign{\vglue 1pt} \hline
    $A_1^{3}$ & 8 & 540 & $3A_1$\\
    \noalign{\vglue 1pt} \hline
    $A_1 A_2$ & 12 & 720 & $A_2+A_1$\\
    \noalign{\vglue 1pt} \hline
    $A_3$ & 24 & 270 & $A_3$\\
    \noalign{\vglue 1pt} \hline
    $A_1^{2} A_2$ & 24 & 1080 & $A_2+2A_1$\\
    \noalign{\vglue 1pt} \hline
    $A_2^{2}$ & 36 & 120 & $2A_2$\\
    \noalign{\vglue 1pt} \hline
    $A_1 A_3$ & 48 & 540 & $A_3+A_1$\\
    \noalign{\vglue 1pt} \hline
    $A_4$ & 120 & 216 & $A_4$\\
    \bottomrule
  \end{tabular}
  \qquad
  \begin{tabular}[t]{crrc}\toprule
    Type of $R$ & $\Size{R}$ & $\Size{[R]}$ & Class 
    \\[4pt] \toprule
    $D_4$ & 192 & 45 & $D_4$\\
    $A_1^{4}$ & 16 & 135 & $4A_1$\\
    \noalign{\vglue 1pt} \hline
    $A_1 A_2^{2}$ & 72 & 360 & $2A_2+A_1$\\
    \noalign{\vglue 1pt} \hline
    $A_1 A_4$ & 240 & 216 & $A_4+A_1$\\
    \noalign{\vglue 1pt} \hline
    $A_5$ & 720 & 36 & $A_5$\\
    \noalign{\vglue 1pt} \hline
    $D_5$ & 1920 & 27 & $D_5$\\
    $A_1^{2} A_3$ & 96 & 270 & $A_3+2A_1$\\
    \noalign{\vglue 1pt} \hline
    $E_6$ & 51840 & 1 & $E_6$\\
    $A_1 A_5$ & 1440 & 36 & $A_5+A_1$\\
    $A_2^{3}$ & 216 & 40 & $3A_2$\\
    \bottomrule
  \end{tabular}
  \bigskip
\end{table}

\begin{table}[p]
  \bigskip
  \small
  \extrarowheight1pt
  \caption{Reflection subgroups of $W(E_7)$.}
  \label{tab:e7}
  \begin{tabular}[t]{crrc}\toprule
    Type of $R$ & $\Size{R}$ & $\Size{[R]}$ & Class 
    \\[4pt] \toprule
    $\emptyset$ & 1 & 1 & $1$\\
    \noalign{\vglue 1pt} \hline
    $A_1$ & 2 & 63 & $A_1$\\
    \noalign{\vglue 1pt} \hline
    $A_1^{2}$ & 4 & 945 & $2A_1$\\
    \noalign{\vglue 1pt} \hline
    $A_2$ & 6 & 336 & $A_2$\\
    \noalign{\vglue 1pt} \hline
    $A_1^{3}$ & 8 & 315 & $3A_1'$\\
    \noalign{\vglue 1pt} \hline
    $A_1^{3}$ & 8 & 3780 & $3A_1''$\\
    \noalign{\vglue 1pt} \hline
    $A_1 A_2$ & 12 & 5040 & $A_2+A_1$\\
    \noalign{\vglue 1pt} \hline
    $A_3$ & 24 & 1260 & $A_3$\\
    \noalign{\vglue 1pt} \hline
    $A_1^{4}$ & 16 & 3780 & $4A_1'$\\
    \noalign{\vglue 1pt} \hline
    $A_1^{2} A_2$ & 24 & 15120 & $A_2+2A_1$\\
    \noalign{\vglue 1pt} \hline
    $A_2^{2}$ & 36 & 3360 & $2A_2$\\
    \noalign{\vglue 1pt} \hline
    $A_1 A_3$ & 48 & 1260 & $A_3+A_1'$\\
    \noalign{\vglue 1pt} \hline
    $A_1 A_3$ & 48 & 7560 & $A_3+A_1''$\\
    \noalign{\vglue 1pt} \hline
    $A_4$ & 120 & 2016 & $A_4$\\
    \noalign{\vglue 1pt} \hline
    $D_4$ & 192 & 315 & $D_4$\\
    $A_1^{4}$ & 16 & 945 & $4A_1''$\\
    \noalign{\vglue 1pt} \hline
    $A_1^{3} A_2$ & 48 & 5040 & $A_2+3A_1$\\
    \noalign{\vglue 1pt} \hline
    $A_1 A_2^{2}$ & 72 & 10080 & $2A_2+A_1$\\
    \noalign{\vglue 1pt} \hline
    $A_1^{2} A_3$ & 96 & 7560 & $A_3+2A_1'$\\
    \noalign{\vglue 1pt} \hline
    $A_2 A_3$ & 144 & 5040 & $A_3+A_2$\\
    \noalign{\vglue 1pt} \hline
    $A_1 A_4$ & 240 & 6048 & $A_4+A_1$\\
    \noalign{\vglue 1pt} \hline
    $A_1 D_4$ & 384 & 945 & $D_4+A_1$\\
    $A_1^{5}$ & 32 & 2835 & $5A_1$\\
    \bottomrule
  \end{tabular}
  \qquad
  \begin{tabular}[t]{crrc}\toprule
    Type of $R$ & $\Size{R}$ & $\Size{[R]}$ & Class 
    \\[4pt] \toprule
    $A_5$ & 720 & 336 & $A_5'$\\
    \noalign{\vglue 1pt} \hline
    $A_5$ & 720 & 1008 & $A_5''$\\
    \noalign{\vglue 1pt} \hline
    $D_5$ & 1920 & 378 & $D_5$\\
    $A_1^{2} A_3$ & 96 & 3780 & $A_3+2A_1''$\\
    \noalign{\vglue 1pt} \hline
    $A_1 A_2 A_3$ & 288 & 5040 & $A_3+A_2+A_1$\\
    \noalign{\vglue 1pt} \hline
    $A_2 A_4$ & 720 & 2016 & $A_4+A_2$\\
    \noalign{\vglue 1pt} \hline
    $A_1 A_5$ & 1440 & 1008 & $A_5+A_1'$\\
    \noalign{\vglue 1pt} \hline
    $A_1 D_5$ & 3840 & 378 & $D_5+A_1$\\
    $A_1^{3} A_3$ & 192 & 3780 & $A_3+3A_1$\\
    \noalign{\vglue 1pt} \hline
    $A_6$ & 5040 & 288 & $A_6$\\
    \noalign{\vglue 1pt} \hline
    $D_6$ & 23040 & 63 & $D_6$\\
    $A_1^{2} D_4$ & 768 & 945 & $D_4+2A_1$\\
    $A_3^{2}$ & 576 & 630 & $D_4(a_1)+2A_1$\\
    $A_1^{6}$ & 64 & 945 & $6A_1$\\
    \noalign{\vglue 1pt} \hline
    $E_6$ & 51840 & 28 & $E_6$\\
    $A_1 A_5$ & 1440 & 1008 & $A_5+A_1''$\\
    $A_2^{3}$ & 216 & 1120 & $3A_2$\\
    \noalign{\vglue 1pt} \hline
    $E_7$ & 2903040 & 1 & $E_7$\\
    $A_1 D_6$ & 46080 & 63 & $D_6+A_1$\\
    $A_7$ & 40320 & 36 & $A_7$\\
    $A_2 A_5$ & 4320 & 336 & $A_5+A_2$\\
    $A_1 A_3^{2}$ & 1152 & 630 & $2A_3+A_1$\\
    $A_1^{3} D_4$ & 1536 & 315 & $D_4+3A_1$\\
    $A_1^{7}$ & 128 & 135 & $7A_1$\\
    \bottomrule
  \end{tabular}
  \bigskip
\end{table}

\begin{table}[p]
  \bigskip
  \footnotesize
  \extrarowheight1pt
  \caption{Reflection subgroups of $W(E_8)$.}
  \label{tab:e8}
  \begin{tabular}[t]{crrc}\toprule
    Type of $R$ & $\Size{R}$ & $\Size{[R]}$ & Class 
    \\[4pt] \toprule
    $\emptyset$ & 1 & 1 & $1$\\
    \noalign{\vglue 1pt} \hline
    $A_1$ & 2 & 120 & $A_1$\\
    \noalign{\vglue 1pt} \hline
    $A_1^{2}$ & 4 & 3780 & $2A_1$\\
    \noalign{\vglue 1pt} \hline
    $A_2$ & 6 & 1120 & $A_2$\\
    \noalign{\vglue 1pt} \hline
    $A_1^{3}$ & 8 & 37800 & $3A_1$\\
    \noalign{\vglue 1pt} \hline
    $A_1 A_2$ & 12 & 40320 & $A_2+A_1$\\
    \noalign{\vglue 1pt} \hline
    $A_3$ & 24 & 7560 & $A_3$\\
    \noalign{\vglue 1pt} \hline
    $A_1^{4}$ & 16 & 113400 & $4A_1''$\\
    \noalign{\vglue 1pt} \hline
    $A_1^{2} A_2$ & 24 & 302400 & $A_2+2A_1$\\
    \noalign{\vglue 1pt} \hline
    $A_2^{2}$ & 36 & 67200 & $2A_2$\\
    \noalign{\vglue 1pt} \hline
    $A_1 A_3$ & 48 & 151200 & $A_3+A_1$\\
    \noalign{\vglue 1pt} \hline
    $A_4$ & 120 & 24192 & $A_4$\\
    \noalign{\vglue 1pt} \hline
    $D_4$ & 192 & 3150 & $D_4$\\
    $A_1^{4}$ & 16 & 9450 & $4A_1'$\\
    \noalign{\vglue 1pt} \hline
    $A_1^{3} A_2$ & 48 & 604800 & $A_2+3A_1$\\
    \noalign{\vglue 1pt} \hline
    $A_1 A_2^{2}$ & 72 & 403200 & $2A_2+A_1$\\
    \noalign{\vglue 1pt} \hline
    $A_1^{2} A_3$ & 96 & 453600 & $A_3+2A_1''$\\
    \noalign{\vglue 1pt} \hline
    $A_2 A_3$ & 144 & 302400 & $A_3+A_2$\\
    \noalign{\vglue 1pt} \hline
    $A_1 A_4$ & 240 & 241920 & $A_4+A_1$\\
    \noalign{\vglue 1pt} \hline
    $A_1 D_4$ & 384 & 37800 & $D_4+A_1$\\
    $A_1^{5}$ & 32 & 113400 & $5A_1$\\
    \noalign{\vglue 1pt} \hline
    $A_5$ & 720 & 40320 & $A_5$\\
    \noalign{\vglue 1pt} \hline
    $D_5$ & 1920 & 7560 & $D_5$\\
    $A_1^{2} A_3$ & 96 & 75600 & $A_3+2A_1'$\\
    \noalign{\vglue 1pt} \hline
    $A_1^{2} A_2^{2}$ & 144 & 604800 & $2A_2+2A_1$\\
    \noalign{\vglue 1pt} \hline
    $A_1 A_2 A_3$ & 288 & 604800 & $A_3+A_2+A_1$\\
    \noalign{\vglue 1pt} \hline
    $A_1^{2} A_4$ & 480 & 362880 & $A_4+2A_1$\\
    \noalign{\vglue 1pt} \hline
    $A_3^{2}$ & 576 & 151200 & $2A_3''$\\
    \noalign{\vglue 1pt} \hline
    $A_2 A_4$ & 720 & 241920 & $A_4+A_2$\\
    \noalign{\vglue 1pt} \hline
    $A_2 D_4$ & 1152 & 50400 & $D_4+A_2$\\
    $A_1^{4} A_2$ & 96 & 151200 & $A_2+4A_1$\\
    \noalign{\vglue 1pt} \hline
    $A_1 A_5$ & 1440 & 120960 & $A_5+A_1''$\\
    \noalign{\vglue 1pt} \hline
    $A_1 D_5$ & 3840 & 45360 & $D_5+A_1$\\
    $A_1^{3} A_3$ & 192 & 453600 & $A_3+3A_1$\\
    \noalign{\vglue 1pt} \hline
    $A_6$ & 5040 & 34560 & $A_6$\\
    \noalign{\vglue 1pt} \hline
    $D_6$ & \llap{23040} & 3780 & $D_6$\\
    $A_1^{2} D_4$ & 768 & 56700 & $D_4+2A_1$\\
    $A_3^{2}$ & 576 & 37800 & $2A_3'$\\
    $A_1^{6}$ & 64 & 56700 & $6A_1$\\
    \bottomrule
  \end{tabular}
  \qquad
  \begin{tabular}[t]{crrc}\toprule
    Type of $R$ & $\Size{R}$ & $\Size{[R]}$ & Class 
    \\[4pt] \toprule
    $E_6$ & 51840 & 1120 & $E_6$\\
    $A_1 A_5$ & 1440 & 40320 & $A_5+A_1'$\\
    $A_2^{3}$ & 216 & 44800 & $3A_2$\\
    \noalign{\vglue 1pt} \hline
    $A_1 A_2 A_4$ & 1440 & 241920 & $A_4+A_2+1$\\
    \noalign{\vglue 1pt} \hline
    $A_3 A_4$ & 2880 & 120960 & $A_4+A_3$\\
    \noalign{\vglue 1pt} \hline
    $A_1 A_6$ & 10080 & 34560 & $A_6+A_1$\\
    \noalign{\vglue 1pt} \hline
    $A_2 D_5$ & 11520 & 30240 & $D_5+A_2$\\
    $A_1^{2} A_2 A_3$ & 576 & 302400 & $A_3+A_2+2A_1$\\
    \noalign{\vglue 1pt} \hline
    $A_7$ & 40320 & 8640 & $A_7''$\\
    \noalign{\vglue 1pt} \hline
    $A_1 E_6$ & 103680 & 3360 & $E_6+A_1$\\
    $A_1^{2} A_5$ & 2880 & 120960 & $A_5+2A_1$\\
    $A_1 A_2^{3}$ & 432 & 134400 & $3A_2+A_1$\\
    \noalign{\vglue 1pt} \hline
    $D_7$ & 322560 & 1080 & $D_7$\\
    $A_1^{2} D_5$ & 7680 & 22680 & $D_5+2A_1$\\
    $A_3 D_4$ & 4608 & 37800 & $D_4+A_3$\\
    $A_1^{4} A_3$ & 384 & 113400 & $A_3+4A_1$\\
    \noalign{\vglue 1pt} \hline
    $E_7$ & 2903040 & 120 & $E_7$\\
    $A_1 D_6$ & 46080 & 7560 & $D_6+A_1$\\
    $A_7$ & 40320 & 4320 & $A_7'$\\
    $A_2 A_5$ & 4320 & 40320 & $A_5+A_2$\\
    $A_1 A_3^{2}$ & 1152 & 75600 & $2A_3+A_1$\\
    $A_1^{3} D_4$ & 1536 & 37800 & $D_4+3A_1$\\
    $A_1^{7}$ & 128 & 16200 & $7A_1$\\
    \noalign{\vglue 1pt} \hline
    $E_8$ & \llap{696729600} & 1 & $E_8$\\
    $D_8$ & \llap{5160960} & 135 & $D_8$\\
    $A_8$ & 362880 & 960 & $A_8$\\
    $A_1 A_7$ & 80640 & 4320 & $A_7+A_1$\\
    $A_1 A_2 A_5$ & 8640 & 40320 & $A_5+A_2+A_1$\\
    $A_4^{2}$ & 14400 & 12096 & $2A_4$\\
    $A_3 D_5$ & 46080 & 7560 & $A_7+A_1$\\
    $A_2 E_6$ & 311040 & 1120 & $E_6+A_2$\\
    $A_1 E_7$ & \llap{5806080} & 120 & $E_7+A_1$\\
    $A_1^{2} D_6$ & 92160 & 3780 & $D_6+2A_1$\\
    $D_4^{2}$ & 36864 & 1575 & $2D_4$\\
    $A_1^{2} A_3^{2}$ & 2304 & 37800 & $2A_3+2A_1$\\
    $A_2^{4}$ & 1296 & 11200 & $4A_2$\\
    $A_1^{4} D_4$ & 3072 & 9450 & $D_4+4A_1$\\
    $A_1^{8}$ & 256 & 2025 & $8A_1$\\
    \bottomrule
  \end{tabular}
  \bigskip
\end{table}

\begin{table}[p]
  \bigskip
  \extrarowheight2pt
  \caption{Reflection subgroups of $W(F_4)$.}
  \label{tab:f4}
  \begin{tabular}[t]{crrc}\toprule
    Type of $R$ & $\Size{R}$ & $\Size{[R]}$ & Class 
    \\[5pt] \toprule
    $\emptyset$ & 1 & 1 & $1$\\
    \noalign{\vglue 2pt} \hline
    $A_1$ & 2 & 12 & $A_1$\\
    \noalign{\vglue 2pt} \hline
    $\tilde A_1$ & 2 & 12 & $\tilde A_1$\\
    \noalign{\vglue 2pt} \hline
    $A_1 \tilde A_1$ & 4 & 72 & $A_1+\tilde A_1$\\
    \noalign{\vglue 2pt} \hline
    $A_2$ & 6 & 16 & $A_2$\\
    \noalign{\vglue 2pt} \hline
    $\tilde A_2$ & 6 & 16 & $\tilde A_2$\\
    \noalign{\vglue 2pt} \hline
    $B_2$ & 8 & 18 & $B_2$\\
    $\tilde A_1^{2}$ & 4 & 18 & $2A_1$\\
    $A_1^{2}$ & 4 & 18 & $2A_1$\\
    \noalign{\vglue 2pt} \hline
    $A_2 \tilde A_1$ & 12 & 48 & $A_2+\tilde A_1$\\
    \noalign{\vglue 2pt} \hline
    $A_1 \tilde A_2$ & 12 & 48 & $\tilde A_2+A_1$\\
    \noalign{\vglue 2pt} \hline
    $B_3$ & 48 & 12 & $B_3$\\
    $A_3$ & 24 & 12 & $A_3$\\
    $\tilde A_1 B_2$ & 16 & 36 & $A_3$\\
    $A_1^{2} \tilde A_1$ & 8 & 36 & $2A_1+\tilde A_1$\\
    $\tilde A_1^{3}$ & 8 & 12 & $2A_1+\tilde A_1$\\
    \noalign{\vglue 2pt} \hline
    $C_3$ & 48 & 12 & $C_3$\\
    $\tilde A_3$ & 24 & 12 & $B_2+A_1$\\
    $A_1 B_2$ & 16 & 36 & $B_2+A_1$\\
    $A_1 \tilde A_1^{2}$ & 8 & 36 & $3A_1$\\
    $A_1^{3}$ & 8 & 12 & $3A_1$\\
    \bottomrule
  \end{tabular}
  \qquad
  \begin{tabular}[t]{crrc}\toprule
    Type of $R$ & $\Size{R}$ & $\Size{[R]}$ & Class 
    \\[5pt] \toprule
    $F_4$ & 1152 & 1 & $F_4$\\
    $B_4$ & 384 & 3 & $B_4$\\
    $C_4$ & 384 & 3 & $B_4$\\
    $\tilde D_4$ & 192 & 1 & $C_3+A_1$\\
    $D_4$ & 192 & 1 & $D_4$\\
    $\tilde A_1 B_3$ & 96 & 12 & $D_4$\\
    $A_1 C_3$ & 96 & 12 & $C_3+A_1$\\
    $B_2^{2}$ & 64 & 9 & $D_4(a_1)$\\
    $A_1 \tilde A_3$ & 48 & 12 & $A_3+\tilde A_1$\\
    $A_3 \tilde A_1$ & 48 & 12 & $A_3+\tilde A_1$\\
    $A_2 \tilde A_2$ & 36 & 16 & $A_2+\tilde A_2$\\
    $\tilde A_1^{2} B_2$ & 32 & 18 & $A_3+\tilde A_1$\\
    $A_1^{2} B_2$ & 32 & 18 & $A_3+\tilde A_1$\\
    $A_1^{2} \tilde A_1^{2}$ & 16 & 18 & $4A_1$\\
    $\tilde A_1^{4}$ & 16 & 3 & $4A_1$\\
    $A_1^{4}$ & 16 & 3 & $4A_1$\\
    \bottomrule
  \end{tabular}
  \bigskip
\end{table}

\begin{table}[p]
  \bigskip
  \extrarowheight2pt
  \caption{Reflection subgroups of $W(G_2)$.}
  \label{tab:g2}
  \begin{tabular}{crrc}\toprule
    Type of $R$ & $\Size{R}$ & $\Size{[R]}$ & Class 
    \\[5pt] \toprule
    $\emptyset$ & 1 & 1 & $1$\\
    \noalign{\vglue 1pt} \hline
    $A_1$ & 2 & 3 & $A_1$\\
    \noalign{\vglue 1pt} \hline
    $\tilde A_1$ & 2 & 3 & $\tilde A_1$\\
    \noalign{\vglue 1pt} \hline
    $G_2$ & 12 & 1 & $G_2$\\
    $\tilde A_2$ & 6 & 1 & $A_2$\\
    $A_1 \tilde A_1$ & 4 & 3 & $A_1 + \tilde A_1$\\
    $A_2$ & 6 & 1 & $A_2$\\
    \bottomrule
  \end{tabular}
\end{table}

\begin{table}[p]
  \bigskip
  \extrarowheight2pt
  \caption{Reflection subgroups of $W(H_3)$.}
  \label{tab:h3}
  \begin{tabular}{crrr}\toprule
    Type of $R$ & $\Size{R}$ & $\Size{[R]}$ & Class 
    \\[5pt] \toprule
    $\emptyset$ & 1 & 1 & $1$\\
    \noalign{\vglue 1pt} \hline
    $A_1$ & 2 & 15 & $2$\\
    \noalign{\vglue 1pt} \hline
    $A_1^{2}$ & 4 & 15 & $4$\\
    \noalign{\vglue 1pt} \hline
    $A_2$ & 6 & 10 & $5$\\
    \noalign{\vglue 1pt} \hline
    $I_2(5)$ & 10 & 6 & $3$\\
    \noalign{\vglue 1pt} \hline
    $H_3$ & 120 & 1 & $6$\\
    $A_1^{3}$ & 8 & 5 & $10$\\
    \bottomrule
  \end{tabular}
\end{table}

\begin{table}[p]
  \bigskip
  \extrarowheight2pt
  \caption{Reflection subgroups of $W(H_4)$.}
  \label{tab:h4} 
  \begin{tabular}[t]{crrc}\toprule
    Type of $R$ & $\Size{R}$ & $\Size{[R]}$ & Class 
    \\[5pt] \toprule
    $\emptyset$ & 1 & 1 & $1$\\
    \noalign{\vglue 1pt} \hline
    $A_1$ & 2 & 60 & $2$\\
    \noalign{\vglue 1pt} \hline
    $A_1^{2}$ & 4 & 450 & $4$\\
    \noalign{\vglue 1pt} \hline
    $A_2$ & 6 & 200 & $5$\\
    \noalign{\vglue 1pt} \hline
    $I_2(5)$ & 10 & 72 & $3$\\
    \noalign{\vglue 1pt} \hline
    $A_1 A_2$ & 12 & 600 & $8$\\
    \noalign{\vglue 1pt} \hline
    $I_2(5) A_1$ & 20 & 360 & $7$\\
    \noalign{\vglue 1pt} \hline
    $A_3$ & 24 & 300 & $9$\\
    \bottomrule
  \end{tabular}
  \qquad
  \begin{tabular}[t]{crrc}\toprule
    Type of $R$ & $\Size{R}$ & $\Size{[R]}$ & Class 
    \\[5pt] \toprule
    $H_3$ & 120 & 60 & $6$\\
    $A_1^{3}$ & 8 & 300 & $20$\\
    \noalign{\vglue 1pt} \hline
    $H_4$ & 14400 & 1 & $11$\\
    $H_3 A_1$ & 240 & 60 & $21$\\
    $I_2(5)^{2}$ & 100 & 36 & $26$\\
    $A_4$ & 120 & 60 & $27$\\
    $A_2^{2}$ & 36 & 100 & $32$\\
    $D_4$ & 192 & 25 & $25$\\
    $A_1^{4}$ & 16 & 75 & $34$\\
    \bottomrule
\end{tabular}
\end{table}

\clearpage

\bigskip 
{\bf Acknowledgments}: The authors acknowledge the financial
support of the DFG-priority programme SPP1489 ``Algorithmic and Experimental
Methods in Algebra, Geometry, and Number Theory''.  Part of the research for
this paper was carried out while the authors were staying at the
Mathematical Research Institute Oberwolfach supported by the ``Research in
Pairs'' programme.
The second author wishes to acknowledge support from Science Foundation
Ireland.
We are grateful to Robert Howlett for helpful discussions.
Finally, we thank the referee for helpful comments.


\bibliographystyle{plain}
\bibliography{reflsubs}

\end{document}